\documentclass[a4paper,12pt]{article}
\usepackage{a4wide}
\usepackage{amsmath}
\usepackage{amssymb}
\usepackage{amsthm}
\usepackage{latexsym}
\usepackage{graphicx}
\usepackage[english]{babel}
\usepackage{makeidx}

\newtheorem{obs} [subsection]{Remark}
\newtheorem{exm} [subsection]{Example}

\newtheorem{prop}[subsection]{Proposition}

\newtheorem{teor}[subsection]{Theorem}
\newtheorem{lema}[subsection]{Lemma}
\newtheorem{cor} [subsection]{Corollary}
\newcommand{\Zng}{$\mathbb Z^n$-graded $S$-module}

\def\sdepth{\operatorname{sdepth}}
\def\depth{\operatorname{depth}}
\def\supp{\operatorname{supp}}

\def\lcm{\operatorname{lcm}}
\begin{document}
\selectlanguage{english}
\frenchspacing

\large
\begin{center}
\textbf{Several inequalities regarding sdepth}

Mircea Cimpoea\c s
\end{center}
\normalsize

\begin{abstract}
We give several bounds for $\sdepth_S(I+J)$, $\sdepth_S(I\cap J)$, $\sdepth_S(S/(I+J))$,
$\sdepth_S(S/(I\cap J))$, $\sdepth_S(I:J)$ and $\sdepth_S(S/(I:J))$ where $I,J\subset S=K[x_1,\ldots,x_n]$ are monomial ideals. Also, we give several equivalent forms of Stanley Conjecture for $I$ and $S/I$, where $I\subset S$ is a monomial ideal.

\noindent \textbf{Keywords:} Stanley depth, Stanley conjecture, monomial ideal.

\noindent \textbf{2000 Mathematics Subject
Classification:}Primary: 13P10.
\end{abstract}

\section*{Introduction}

Let $K$ be a field and $S=K[x_1,\ldots,x_n]$ the polynomial ring over $K$.
Let $M$ be a \Zng. A \emph{Stanley decomposition} of $M$ is a direct sum $\mathcal D: M = \bigoplus_{i=1}^rm_i K[Z_i]$ as $K$-vector space, where $m_i\in M$, $Z_i\subset\{x_1,\ldots,x_n\}$ such that $m_i K[Z_i]$ is a free $K[Z_i]$-module. We define $\sdepth(\mathcal D)=\min_{i=1}^r |Z_i|$ and $\sdepth_S(M)=\max\{\sdepth(\mathcal D)|\;\mathcal D$ is a Stanley decomposition of $M\}$. The number $\sdepth(M)$ is called the \emph{Stanley depth} of $M$. Herzog, Vladoiu and Zheng show in \cite{hvz} that this invariant can be computed in a finite number of steps if $M=I/J$, where $J\subset I\subset S$ are monomial ideals. There are two important particular cases. If $I\subset S$ is a monomial ideal, we are interested in computing $\sdepth_S(S/I)$ and $\sdepth_S(I)$ and to find some relation between them. 

Let $I\subset S'=K[x_1,\ldots,x_r]$, $J\subset S''=K[x_{r+1},\ldots,x_n]$ two monomial ideals, and consider $S=K[x_1,\ldots,x_n]$. In Theorem $1.2$, we give some lower and upper bounds for $\sdepth_S(IS+JS)$ and $\sdepth_S(S/(IS\cap JS))$. Some lower bounds for $\sdepth_S(IS\cap JS)$ and $\sdepth_S(S/(IS+JS))$ were given in \cite{apop}, respective in \cite{asia}. An important fact, which will use implicitly in our paper, is that $\sdepth_S(IS)=\sdepth_{S'}(I)+n-r$, see \cite{hvz}. Also, obviously, $\depth_S(IS)=\depth_{S'}(I)+n-r$. In \cite{asia}, A.\ Rauf conjectured that $\sdepth_S(I)\geq \sdepth_S(S/I)+1$. We prove that this inequality holds, if $\sdepth_S(I)=\sdepth_{S[y_1]}(I,y_1)$, see Remark $1.4$. In the first section we also give some corollaries of Theorem $1.1$. 

In section $2$, we consider the more general case, when $I,J\subset S$ are two monomial ideals. In Theorem $2.2$, we give lower bounds for $\sdepth_S(I+J), \sdepth_S(I\cap J), \sdepth_S(S/(I+J))$ and $\sdepth_S(S/(I\cap J))$, where $I,J\subset S$ are two monomial ideals. In section $3$, we prove that if $I\subset S$ is a monomial ideal, and $v\in S$ a monomial, then $\sdepth_S{S/(I:v)}\geq \sdepth_S(S/I)$, see Proposition $2.7$. As a consequence, we give lower bounds for $\sdepth_S(I:J)$ and $\sdepth_S(S/(I:J))$, where $I,J\subset S$ are monomial ideals, see Corollary $2.12$. Also, if $I\subset S$ is a monomial ideal, we give some bounds for $\sdepth_S(I)$ and $\sdepth_S(S/I)$, in terms of the irreducible irredundant decomposition of $I$, see Corollary $2.13$, and in terms of the primary irredundant decomposition of $I$, see Corollary $2.14$.

In section $3$, we give several equivalent forms of Stanley Conjecture for $I$ and $S/I$, where $I\subset S$ is a monomial ideal. See Propositions $3.1$, $3.3$, $3.4$ and $3.8$.

\vspace{3mm} \noindent {\footnotesize
\begin{minipage}[b]{15cm}
 Mircea Cimpoeas, Simion Stoilow Institute of Mathematics of the Romanian Academy\\
 E-mail: mircea.cimpoeas@imar.ro
\end{minipage}}

\newpage
\section{Case of ideals with disjoint support}

We denote $S=K[x_1,\ldots,x_n]$ the ring of polynomials in $n$ variables, where $n\geq 2$. For a monomial $u\in S$, we denote $\supp(u)=\{x_i:\; x_i|u\}$. We begin this section with the following lemma.

\begin{lema}
Let $u,v\in S$ be two monomials and $Z,W\subset \{x_1,\ldots,x_n\}$, such that $\supp(u)\subset W$ and $\supp(v)\subset Z$. Then $uK[Z] \cap vK[W] = \lcm(u,v) K[Z\cap W]$.
\end{lema}

\begin{proof}
"$\supseteq$": Since $\lcm(u,v)=u \cdot (v/gcd(u,v))$ and $\supp(v)\in K[Z]$, it follows that $\lcm(u,v)\in uK[Z]$. Analogously, $\lcm(u,v)\in vK[W]$ and therefore, it follows that $\lcm(u,v)\in uK[Z]\cap vK[W]$.

"$\subseteq$": Let $w\in uK[Z] \cap vK[W]$ be a monomial. It follows that $w = u\cdot a = v\cdot b$, where $a\in K[Z]$ and $b\in K[W]$ are some monomial. Thus $\lcm(u,v)|w$ and $w = \lcm(u,v)\cdot c$, where $c= w/\lcm(u,v)= a/(\lcm(u,v)/u) = b/(\lcm(u,v)/v)$. Therefore, $c\in K[Z]\cap K[W] = K[Z\cap W]$. 
\end{proof}

\begin{teor}
Let $I\subset S'=K[x_1,\ldots,x_r]$, $J\subset S''=K[x_{r+1},\ldots,x_n]$ be monomial ideals, where $1\leq r<n$. Then, we have the following inequalities:

(1) $\sdepth_S(IS) \geq \sdepth_S(IS+JS) \geq \min\{ \sdepth_S(IS), \sdepth_{S''}(J) + \sdepth_{S'}(S'/I) \}$.

(2) $\sdepth_S(IS\cap JS) \geq \sdepth_{S'}(I) + \sdepth_{S''}(J)$.

(3) $\sdepth_S(S/IS) \geq \sdepth_S(S/(IS\cap JS)) \geq \min \{ \sdepth_S(S/IS),  \sdepth_{S''}(S''/J) + \sdepth_{S'}(I)\}$.

(4) $\sdepth_S(S/(IS+JS)) \geq \sdepth_{S'}(S'/I) + \sdepth_{S''}(S''/J)$.

(5) $\depth_S(S/(IS\cap JS))-1 = \depth_S(S/(IS+JS)) = \depth_{S'}(S'/I)+\depth_{S''}(S''/J)$.

(6) $\depth_S(IS \cap JS) = \depth_S(IS+JS)+1 = 
\depth_{S'}(I) + \depth_{S''}(J)$ and \linebreak $\depth_S((IS+JS)/IS)= \depth_S(IS+JS)$.
\end{teor}

\begin{proof}
(1) For the first inequality, let $IS + JS = \bigoplus_{i=1}^r w_iK[W_i]$ be a Stanley decomposition of the ideal $IS+ JS\subset S$. Note that $(IS+JS)\cap S' = IS\cap S' = I$, since $JS\cap S'=(0)$. Therefore, $I=\bigoplus_{i=1}^r (w_iK[W_i]\cap S')$. If $w_i\in S'$, we have $w_iK[W_i]\cap S' = w_iK[W_i\cap \{x_1,\ldots,x_r\}]$, by Lemma $1.1$. On the other hand, if $w_i\notin S'$, we have $w_iK[W_i]\cap S' = (0)$. Thus, $I = \bigoplus_{w_i\in S'} w_i K[W_i\cap \{x_1,\ldots,x_r\}]$. It follows that $IS = \bigoplus_{w_i\in S'} w_i K[W_i\cup \{x_{r+1},\ldots,x_{n}\}]$. Therefore, 
$\sdepth_{S}(IS+JS)\leq \sdepth_{S}(IS)$.

In order to prove the second inequality, we consider the Stanley decompositions $S'/I=\bigoplus_{i=1}^r u_iK[U_i]$ and $J = \bigoplus_{j=1}^s v_jK[V_j]$. It follows that $S/IS=\bigoplus_{i=1}^r u_iK[U_i\cup\{x_{r+1},\ldots, x_n\}]$ and
$JS = \bigoplus_{j=1}^s v_jK[V_j\cup \{x_1,\ldots,x_r\}]$ are Stanley decompositions for $S/IS$, respectively for $JS$. We consider the decomposition: 
\[(*)\;\;\;\; IS+JS = ((IS+JS)\cap IS) \oplus ((IS+JS)\cap S/IS) = IS \oplus (JS\cap S/IS).\]
We have $JS\cap S/IS = \bigoplus_{i=1}^r \bigoplus_{j=1}^s u_iK[U_i\cup\{x_{r+1},\ldots, x_n\}] \cap v_j K[V_j\cup \{x_{r+1},\ldots, x_n\}]$. Since $u_i\in S'$ and $v_j\in S''$ for all $(i,j)'s$, by Lemma $1.1$, it follows that $JS\cap S/IS = \bigoplus_{i=1}^r \bigoplus_{j=1}^s u_iv_j K[U_i\cup V_j]$ and therefore 
$\sdepth_S(JS\cap S/IS)\geq \sdepth_S''(J)$. Thus, by $(*)$, we get the required conclusion.

(2) It was proved in \cite[Lemma 1.1]{apop}.

(3) For the first inequality, let $S/(IS+JS)= \bigoplus_{i=1}^r w_iK[W_i]$ be a Stanley decomposition of $S/(IS+JS)$. As in the proof of (1), we get $S/IS = \bigoplus_{w_i\in S'} w_iK[W_i \cup \{x_{r+1},\ldots x_n\}]$ and thus we get $\sdepth_S(S/IS) \geq \sdepth_S(S/(IS\cap JS))$. In order to prove the second inequality, we consider the decomposition:
\[ S/(IS\cap JS) = (S/(IS\cap JS) \cap S/IS) \oplus (S/(IS\cap JS) \cap IS ) = S/IS \oplus ((S/JS) \cap IS) \]
and, as in the proof of $(1)$, we get $\sdepth_S ((S/JS) \cap IS) \geq \sdepth_{S'}(I) + \sdepth_{S''}(S''/J)$ and thus we obtain the required conclusion.

(4) It was proved in \cite[Theorem 3.1]{asia}.

(5) It is a consequence of Depth's Lemma for the short exact sequence of $S$-modules
\[ 0 \rightarrow S/(IS\cap JS) \rightarrow S/IS \oplus S/JS \rightarrow S/(IS+JS) \rightarrow 0. \]
See also \cite[Lemma 1.1]{apop} for more details.

(6) The first equality is a direct consequence of (5). The second follows by Depth Lemma for the short exact sequence $0\rightarrow I\rightarrow I+J \rightarrow (I+J)/I \rightarrow 0$.
\end{proof}

\begin{obs}
If $I\subset S$ is a monomial ideal, we define \emph{the support of $I$} to be the set $\supp(I)=\bigcup_{u\in G(I)}\supp(u)$, where $G(I)$ is the set on minimal monomial generators of $I$. With this notation, we can reformulate Theorem $1.2$ in terms of two monomial ideals $I,J\subset S$ with $\supp(I)\cap \supp(J)=\emptyset$. The conclusions should be also modified, as follows. If $I,J\subset S$ are two monomial ideals with disjoint supports, then $\sdepth_S(I\cap J)\geq \sdepth_S(I)+\sdepth_S(J)-n$ etc.

With the above notations, we may consider the short exact sequences $0 \rightarrow I \rightarrow I+J \rightarrow (I+J)/I \rightarrow 0$ and $0 \rightarrow I/(I\cap J) \cong (I+J)/J \rightarrow S/(I\cap J) \rightarrow S/J \rightarrow 0$. It follows that $\sdepth_S(I+J)\geq \min\{\sdepth_S(I),\sdepth_S((I+J)/I)\}$ and $\sdepth_S(S/(I\cap J))\geq \min\{\sdepth_S(S/I),\sdepth_S((I+J)/J)\}$. Note that $(I+J)/I = J\cap (S/I)$ and $(I+J)/J = I\cap (S/J)$. From the proof of Theorem $1.2(1)$, we get $\sdepth_S((I+J)/I)\geq \sdepth_S(J) + \sdepth_S(S/I) - n$, if $\supp(I)\cap \supp(J)=\emptyset$.
\end{obs}

We recall the facts that if $I=(u_1,\ldots,u_m)\subset S$ is a monomial complete intersection, then $\sdepth_S(I)=n-\left\lfloor m/2 \right\rfloor$, see \cite[Theorem 2.4]{shen} and $\sdepth_S(S/I) = n - m$, see \cite[Theorem 1.1]{asia1}. On the other hand, if $I=(u_1,\ldots,u_m)\subset S$ is an arbitrary monomial ideal,
then, according to \cite[Theorem 2.1]{okazaki}, $\sdepth_S(I)\geq n-\left\lfloor m/2 \right\rfloor$ and
according to \cite[Proposition 1.2]{mir}, $\sdepth_S(S/I)\geq n-m$. Using these results, we proved the following:

\begin{cor}
Let $I\subset S'=K[x_1,\ldots,x_r]$ be a monomial ideal and $J=(u_1,\ldots,u_m)\subset S''=K[x_{r+1},\ldots,x_n]$ 
be a monomial ideal. Then:

(1) $\sdepth_S(IS) \geq \sdepth_S(IS+JS) \geq \min\{ \sdepth_S(IS), 
     \sdepth_{S}(S/SI) - \left\lfloor m/2 \right\rfloor \}$.
     
(2) $\sdepth_S(IS\cap JS) \geq \sdepth_{S}(IS) - \left\lfloor m/2 \right\rfloor $.

(3) $\sdepth_S(S/IS) \geq \sdepth_S(S/(IS\cap JS)) \geq \min \{ \sdepth_S(S/IS), \sdepth_{S}(IS) - m\}$.

(4) $\sdepth_S(S/(IS+JS)) \geq \sdepth_{S}(S/IS) - m$.

(5) In particular, if $J$ is complete intersection, then:
$\depth_S(S/(IS\cap JS))-1 = \depth_S(S/(IS+JS)) = \depth_{S}(S/IS) - m$.
\end{cor}


\begin{obs}
\emph{Let $I\subset S = K[x_1,\ldots,x_n]$ be a monomial ideal. If we denote $\bar{S}=S[y_1,\ldots,y_m]$, then, by Corollary $1.4(1)$, we have
\[ \sdepth_S(I)+m \geq \sdepth_{\bar{S}}(I,y_1,\ldots,y_m) \geq \min\{ \sdepth_{S}(I) + m , \sdepth_S(S/I) + \left\lceil  m/2 \right\rceil \}.\]}
\emph{Assume $\sdepth_S(I)+m > \sdepth_{\bar{S}}(I,y_1,\ldots,y_m)$. It follows that $\sdepth_{S}(I) + m > \sdepth_S(S/I) + \left\lceil m/2 \right\rceil $ and therefore 
$\sdepth_S(I) \geq \sdepth_S(S/I) + \left\lfloor m/2 \right\rfloor + 1$.
In particular, if $m=1$ and $\sdepth_{\bar{S}}(I,y_1)=\sdepth_S(I)$, then $\sdepth_S(I)\geq \sdepth_S(S/I) + 1$ and thus we get a positive answer to the problem put by Asia in \cite{asia}.}
\end{obs}

\begin{cor}
With the notations of Theorem $1.2$, we have the followings:

(1) If the Stanley conjecture hold for $I$ and $J$, then the Stanley conjecture holds for $IS\cap JS$.

(2) If the Stanley conjecture hold for $S'/I$ and $S''/J$, then the Stanley conjecture holds for $S/(IS+JS)$.

(3) If the Stanley conjecture hold for $J$ and $S'/I$ or for $I$ and $S''/J$, then the Stanley conjecture hold
for $(IS+JS)$ and $S/(IS\cap JS)$. 
\end{cor}

\begin{proof}
(1) It is a direct consequence of Theorem $1.2(2)$ and $1.2(6)$. (2) It is a direct consequence of Theorem $1.2(4)$ and $1.2(5)$.

(3) Assume the Stanley conjecture hold for $J$ and $S'/I$. According to Theorem $1.2(1)$, we have
$\sdepth_S(IS+JS)\geq \min\{ \sdepth_S(IS), \sdepth_{S''}(J) + \sdepth_{S'}(S'/I) \}$. \linebreak If $\sdepth_S(IS+JS)=\sdepth_S(IS)$, then, by $1.2(6)$, we get $\sdepth_S(IS+JS)\geq \depth_S(IS) = \depth_{S'}(I)+n-r \geq \depth_{S'}(I)+\depth_{S''}(J) > \depth_S(IS+JS)$.

If $\sdepth_S(IS+JS)<\sdepth_S(IS)$, it follows that $\sdepth_S(IS+JS)\geq \sdepth_{S''}(J) + \sdepth_{S'}(S'/I) \geq \depth_{S''}(J) + \depth_{S'}(S'/I) = \depth_S(IS+JS)$. In the both cases, the ideal $IS+JS$ satisfies the Stanley conjecture. The case when $I$ and $S''/J$ satisfy the Stanley conjecture is similar. Also, the proof of the fact that $S/(IS\cap JS)$ satisfies the Stanley conjecture follows in the same way from $1.2(3)$ and $1.2(5)$.
\end{proof}

Note that, by the proof of Corollary $1.6(1)$, if $\sdepth_S(IS+JS)=\sdepth_S(IS)$, then $\sdepth_S(IS+JS) \geq \depth_S(IS+JS)+ n-r-\depth_{S''}(S''/J)$. Analogously, if \linebreak $\sdepth_S(S/(IS\cap JS)) = \sdepth_S(IS)$ then $\sdepth_S(S/(IS\cap JS)) \geq \depth_S(S/(IS\cap JS))+n-r-\depth_{S''}(S''/J)$.

\begin{cor}
Let $I_j\subset S_j:=[x_{j1},\ldots,x_{jn_j}]$ be some monomial ideals, where $k\geq 2$,$n_j\geq 1$ and $1\leq j\leq k$. Denote $S=K[x_{ji}:\;1\leq j\leq k,\; 1\leq i\leq n_j]$. Then, the following inequalities hold:

(1) $\sdepth_S(I_1S\cap \cdots \cap I_kS) \geq \sdepth_{S_1}(I_1) + \cdots + \sdepth_{S_k}(I_k)$.

(2) $\sdepth_S(I_1S + \cdots +I_kS) \geq \min \{\sdepth_{S_1}(I_1)+n_2+\cdots+n_k, \sdepth_{S_2}(I_2)+\sdepth_{S_1}(S_1/I_1) + n_3 + \cdots + n_k, \ldots, \sdepth_{S_k}(I_k) + \sdepth_{S_{k-1}}(S_{k-1}/I_{k-1}) + \cdots + \sdepth_{S_{1}}(S_{1}/I_{1}) \}$.

$\sdepth_S(I_1S + \cdots +I_kS) \leq \min\{\sdepth_S(I_jS):\;j=1,\ldots,k\}$.

(3) $\sdepth_S(S/(I_1S\cap \cdots \cap I_kS)) \geq \min \{\sdepth_{S_1}(S_1/I_1)+n_2+\cdots+n_k, \sdepth_{S_2}(S_2/I_2)+\sdepth_{S_1}(I_1) + n_3 + \cdots + n_k, \ldots, \sdepth_{S_k}(S_k/I_k) + \sdepth_{S_{k-1}}(I_{k-1}) + \cdots + \sdepth_{S_{1}}(I_{1})\}$

$\sdepth_S(S/(I_1S\cap \cdots \cap I_kS)) \leq \min\{\sdepth_S(S/I_jS):\;j=1,\ldots,k \}$.

(4) $\sdepth_S(S/(I_1S+\cdots+I_kS)) \geq \sdepth_{S_1}(I_1S) + \cdots + \sdepth_{S_k}(I_kS)$.

(5) $\depth_S(I_1S\cap \cdots \cap I_kS) = \depth_S(I_1S+\cdots+I_kS)+(k-1) = \depth_{S_1}(I_1) + \cdots + \depth_{S_k}(I_k)$.
\end{cor}

\begin{proof}
We use induction on $k\geq 2$ and we apply Theorem $1.2$.
\end{proof}

\begin{cor}
With the notation of the previous Corollary, we have:

(1) If $I_1,\ldots,I_k$ satisfy the Stanley Conjecture, then $I_1S \cap \cdots \cap I_kS$ satisfies the Stanley Conjecture.

(2) If $1\leq l\leq n$ is an integer and the Stanley conjecture holds for $I_l$ and $S/I_j$ for all $j\neq l$ then, the Stanley Conjecture holds for $I_1S+\cdots+I_kS$.

(3) If $1\leq l\leq n$ is an integer and the Stanley conjecture holds for $S_l/I_l$ and $I_j$ for all $j\neq l$ then, the Stanley Conjecture holds for $S/(I_1S\cap \cdots \cap I_kS)$.

(4) If $S/I_1,\ldots,S/I_k$ satisfy the Stanley Conjecture, then $S/(I_1S + \cdots + I_kS)$ satisfies the Stanley Conjecture.
\end{cor}

\begin{proof}
(1) We use induction on $k$ and apply Corollary $1.7(1)$.

(2) We may assume $l=k$. Denote $S'=K[x_{ji}:\;1\leq j\leq k-1,\;1\leq i\leq n_j]$ and consider the ideal $I':=I_1S'+\cdots+I_{k-1}S'\subset S$.  By $(1)$, it follows that the Stanley Conjecture holds for $S'/I'$. 
We denote $I=I_1S+\cdots+I_kS$. According to Corollary $1.7(3)$, since Stanley conjecture holds for $S'/I'$ and $I_{k}$ and since $I=I'S+I_kS$, it follows that the Stanley Conjecture holds for $I$.

(3) The proof is similar to the proof of (2).

(4) We use induction on $k$ and apply Corollary $1.7(4)$.
\end{proof}

\begin{cor}
With the notations of $1.7$, if all $n_j\leq 5$ and all $I_j's$ are squarefree, then $I_1S\cap \cdots \cap I_kS$, $I_1S +\cdots+I_kS$, $S/(I_1S \cap \cdots \cap I_kS)$ and $S/(I_1S + \cdots + I_kS)$ satisfy the Stanley Conjecture.
\end{cor}

\begin{proof}
Indeed, if $I\subset K[x_1,\ldots,x_n]$ is a squarefree monomial ideal with $n\leq 5$, then both $I$ and $S/I$ satisfies the Stanley Conjecture, see \cite{pop} and \cite{pope}. Therefore, $I_j's$ and $S_j/I_j's$ satisfy the Stanley Conjecture. By Corollary $1.8$ we are done.
\end{proof}

\begin{exm}
Let $I=(x_{11},\ldots,x_{1n_1})\cap (x_{21},\ldots,x_{2n_2})\cap \cdots \cap (x_{k1},\ldots,x_{kn_k}) \subset S$, 
where $k\geq 2$,$n_j\geq 1$, $1\leq j\leq k$ and $S=K[x_{ji}:\;1\leq j\leq k,\; 1\leq i\leq n_j]$. 
According to Corollary $1.7(1)$, $\sdepth_S(I)\geq \left\lceil n_1/2 \right\rceil + \cdots + \left\lceil n_k/2 \right\rceil$. Note that $\sdepth_S(I)\geq \depth_S(I) = k$. Also, according to Corollary $3.2$ or \cite[Theorem 3.1]{ishaq}, $\sdepth_S(I) \leq \min\{n- \left\lfloor n_j/2 \right\rfloor:\;1\leq j\leq k\}$. 

Now, we want to estimate $\sdepth_S(S/I)$. According to Corollary $1.7(3)$, we have:
\[ \sdepth_S(S/I) \geq \min\{ n_2+\cdots+n_k, \left\lceil n_1/2 \right\rceil + n_3 + \cdots + n_k, 
\left\lceil n_1/2 \right\rceil + \]
\[ + \left\lceil n_2/2 \right\rceil + n_4 + \cdots +n_k, \ldots, \left\lceil n_1/2 \right\rceil + \cdots 
+ \left\lceil n_{k-1}/2 \right\rceil + n_k \} \] 

Note that $\sdepth_S(S/I)\geq \depth_S(S/I)=k-1$. Also, according to Corollary $3.2$ or Corollary $1.7(3)$, we have $\sdepth_S(S/I)\leq \min\{n-n_j:\;1\leq j\leq k\}$. 
\end{exm}

\section{The general case}

In the following, we consider $1\leq s\leq r+1 \leq n$ three integers, with $n\geq 2$. We denote $S':=K[x_1,\ldots,x_r]$, $S'':=K[x_s,\ldots,x_n]$ and $S:=K[x_1,\ldots,x_n]$. Let $p:=r-s+1$.

\begin{lema}
Let $u\in S'$ and $v\in S''$ be two monomials, $Z\subset \{x_1,\ldots,x_r\}$ and $W\subset \{x_s,\ldots,x_n\}$ two subsets of variables. We denote $\bar{Z}:=Z\cup\{x_{r+1},\ldots,x_n\}$ and $\bar{W}:=W\cup\{x_1,\ldots,x_{s-1}\}$.
If $L:=uK[\bar{Z}] \cap vK[\bar{W}]$, then $L=\{0\}$ or $L = \lcm(u,v)K[(Z\cup W) \setminus Y]$, where $Y\subset \{x_s,\ldots,x_r\}$ and with $|(Z\cup W)\setminus Y|\geq |Z|+|W|-p$.
\end{lema}

\begin{proof}
We use induction on $p=r-s+1$. If $p=0$, it follows that $s=r+1$ and therefore $\supp(u)\subset \{x_1,\ldots,x_{s-1}\}$ and $\supp(v)\subset \{x_{r+1},\ldots,x_n\}$. Thus, by Lemma $1.1$, we get
\[ L = \lcm(u,v)K[(Z\cup\{x_{r+1},\ldots,x_n\}) \cap (W\cup\{x_1,\ldots,x_{r}\}) ] = \lcm(u,v)K[Z\cup W].\]

Now, assume $p>0$, i.e. $r\geq s$. We must consider several cases. First, suppose $x_s \notin \supp(u)$ and $x_s \notin \supp(v)$. If $x_s\in Z\cap W$, we can write $L = uK[\bar{Z}]\cap vK[\bar{W}] = (uK[\bar{Z}\setminus \{x_s\}]\cap vK[\bar{W}\setminus \{x_s\}])[x_s]$. Using the induction hypothesis, we are done. On the other hand, if $x_s\notin Z\cap W$, then $L = uK[\bar{Z}\setminus \{x_s\}]\cap vK[\bar{W}\setminus \{x_s\}]$. Note that 
$|\bar{Z}\cap \bar{W}| = |\bar{Z}\setminus \{x_s\} \cap \bar{W}\setminus \{x_s\}| \geq |W\setminus \{x_s\}| + |Z\setminus \{x_s\}|-p+1 \geq |Z| + |W| - p$, since the variable $x_s$ appear only in one of the sets $W$ and $Z$. Therefore, by induction, we are done.

Now, assume $x_s\in\supp(u)$, and denote $\alpha=\max\{j:\;x_s^j|u\}$ and $\beta=\max\{j:\;x_s^j|v\}$. We write 
$u=x_s^{\alpha}\tilde{u}$ and $v=x_s^{\beta}\tilde{v}$. If $x_s\notin Z$ we have two subcases:

a) Assume $x_s\notin W$. If $\alpha\neq \beta$, it follows that $L=\{0\}$. If $\alpha=\beta$, then $L=x_s^{\alpha}(\tilde{u}K[Z]\cap \tilde{v}K[W])$ and we are done by induction, noting that $\lcm(u,v)=x_s^{\alpha}\lcm(\tilde{u},\tilde{v})$.

b) If $x_s\in W$ and $\alpha<\beta$, we have $L=\{0\}$. If $\alpha\geq\beta$, we have $L=x_s^{\alpha}(\tilde{u}K[Z]\cap \tilde{v}K[W])$ and we are done by induction, noting that $\lcm(u,v)=x_s^{\alpha}\lcm(\tilde{u},\tilde{v})$.

If $x_s\in Z$, we must also consider two subcases:

a) If $x_s\notin W$ and $\alpha>\beta$, it follows that $L=\{0\}$. If $\alpha\leq \beta$, we have $L=x_s^{\beta}(\tilde{u}K[Z]\cap \tilde{v}K[W])$ and we are done by induction.

b) If $x_s\in W$, we have $L=x_s^{\max\{\alpha,\beta\}}(\tilde{u}K[\bar{Z}\setminus \{x_s\}]\cap \tilde{v}K[\bar{W}\setminus \{x_s\}])[x_s]$ and, again, we are done by induction.
\end{proof}

Now, we are able to prove the following theorem, which generalize some results of Theorem $1.2$.

\begin{teor}
Let $I\subset S'$ and $J\subset S''$ be two monomial ideals. Then:

(1) $\sdepth_S(IS\cap JS) \geq \sdepth_{S'}(I) + \sdepth_{S''}(J) - p = \sdepth_{S}(IS) + \sdepth_{S}(JS) - n$.

(2) $\sdepth_S(S/(IS+JS)) \geq \sdepth_{S'}(S'/I) + \sdepth_{S''}(S''/J) - p = \sdepth_{S}(S/IS) + \sdepth_{S}(S/JS) - n$.

(3) $\sdepth_S(IS+JS)\geq \min\{\sdepth_S(IS), \sdepth_{S''}(J) + \sdepth_{S'}(S'/I)-p\} = \linebreak = \min\{\sdepth_S(IS), \sdepth_{S}(JS) + \sdepth_{S}(S/IS)-n\}$.

(4) $\sdepth_S(S/(IS\cap JS)) \geq \min \{\sdepth_S(S/IS), \sdepth_{S''}(S''/J) + \sdepth_{S'}(I)-p\} = \linebreak =
 \min \{\sdepth_S(S/IS), \sdepth_{S}(S/JS) + \sdepth_{S}(IS)-n\}$.
\end{teor}

\begin{proof}
(1) We consider $I=\bigoplus_{i=1}^a u_iK[Z_i]$ and $J=\bigoplus_{j=1}^b v_jK[W_j]$ two Stanley decomposition for $I$, respective for $J$. Then $IS=\bigoplus_{i=1}^a u_iK[\bar{Z}_i]$, where $\bar{Z}_i=Z_i\cup \{x_{r+1},\ldots,x_n\}$ and 
$JS = \bigoplus_{j=1}^b v_iK[\bar{W}_i]$, where $\bar{W}_j = W_j\cup \{x_1,\ldots,x_{s-1}\}$. We have $IS\cap JS=\bigoplus_{i=1}^a\bigoplus_{j=1}^b L_{ij}$ a Stanley decomposition for $IS\cap JS$, where $L_{ij}:=u_iK[\bar{Z}_i]\cap v_j[\bar{W}_j]$. According to Lemma $2.1$, $L_{ij}=\{0\}$ or $L_{ij}=\lcm(u_i,v_j)K[(Z_i\cup W_j)\setminus Y_{ij}]$, where $Y_{ij}\subset \{x_s,\ldots,x_r\}$ and $|(Z_i\cup W_j)\setminus Y_{ij}| \geq |Z_i| + |W_j| - p$. Therefore, we are done.

(2) The proof is similar with the proof of (1).

(3) We consider $S'/I=\bigoplus_{i=1}^a u_iK[Z_i]$ and $J=\bigoplus_{j=1}^b v_jK[W_j]$ two Stanley decomposition for $S'/I$, respective for $J$. Then $S/IS=\bigoplus_{i=1}^a u_iK[\bar{Z}_i]$, where $\bar{Z}_i=Z_i\cup \{x_{r+1},\ldots,x_n\}$ and $JS = \bigoplus_{j=1}^b v_iK[\bar{W}_i]$, where $\bar{W}_j = W_j\cup \{x_1,\ldots,x_{s-1}\}$. We use the decomposition:
\[ IS+JS = ((IS+JS)\cap IS) \oplus ((IS+JS)\cap (S/IS)) = IS \oplus (JS\cap (S/IS)). \]
If follows, that $\sdepth_S(IS+JS)\geq \min\{\sdepth_{S}(IS), \sdepth_{S}(JS\cap (S/IS))\}$. We have $JS\cap S/IS=\bigoplus_{i=1}^a\bigoplus_{j=1}^b L_ij$ a Stanley decomposition for $IS\cap JS$, where $L_{ij}:=u_iK[\bar{Z}_i]\cap v_j[\bar{W}_j]$. By Lemma $2.1$, it follows that $\sdepth_{S}(JS\cap (S/IS))\geq \sdepth_{S'}(S'/I) +\sdepth_{S''}(J)$ and therefore we are done.

(4) The proof is similar with the proof of (3).
\end{proof}

\begin{obs}
Note that the results of the previous Theorem do not depend on the numbers $r$ and $s$. Therefore, we can reformulate the Theorem $2.2$ in terms of arbitrary monomial ideals $I,J\subset S$. Also, if $I,J\subset S$ are two monomial ideals, the minimal number $p$ which can be chose, by a reordering of the variables, is $p=|\supp(I)\cap \supp(J)|$. 

Also, as in Remark $1.3$, we have $\sdepth_S((I+J)/I) \geq \sdepth_S(J) + \sdepth_S(S/I)-n$. Therefore, in particular, if $I\subset J$, then $\sdepth_S(J/I)\geq \sdepth_S(J)+\sdepth_S(S/I) - n$.
\end{obs}

Using the previous remark, we have the following Corollary.

\begin{cor}
If $I,J\subset S$ are two monomial ideals and $|G(J)|=m$, then:

(1) $\sdepth_S(I\cap J)\geq \sdepth_S(I) - \left\lfloor m/2 \right\rfloor$.

(2) $\sdepth_S(I + J)\geq \min \{ \sdepth_S(I), \sdepth_S(S/I) - \left\lfloor m/2 \right\rfloor \}$.

    $\sdepth_S(I + J)\geq \sdepth_S(I) - m$.

(3) $\sdepth_S(S/(I+J)) \geq \sdepth_S(S/I) - m$.

(4) $\sdepth_S(S/(I\cap J)) \geq \min\{ \sdepth_{S}(S/I), \sdepth_S(I) - m \}$.

    $\sdepth_S(S/(I\cap J)) \geq \min\{ n-m, \sdepth_{S}(S/I) -  \left\lfloor m/2 \right\rfloor\}$.
    
(5) $\sdepth_S((I+J)/I)\geq \sdepth_S(S/I) -  \left\lfloor m/2 \right\rfloor$.
    
    $\sdepth_S((I+J)/J)\geq \sdepth_S(I) - m$.
\end{cor}

\begin{proof}
We apply Theorem $2.2$ and use the facts that $\sdepth_S(J)\geq n-\left\lfloor m/2 \right\rfloor$, see \cite[Theorem 2.1]{okazaki} and $\sdepth_S(S/J)\geq n-m$, see \cite[Proposition 1.2]{mir}.
\end{proof}

\begin{cor}
If $I\subset S$ is a monomial ideal and $u\in S$ a monomial, then:

(1) $\sdepth_S(I\cap(u)) \geq \sdepth_S(I)$.

(2) $\sdepth_S(I,u)\geq \min\{\sdepth_S(I),\sdepth_S(S/I)\}$.

(3) $\sdepth_S(S/(I,u)) \geq \sdepth_S(S/I) - 1$.

(4) $\sdepth_S(S/(I\cap(u))) \geq \sdepth_S(S/I)$. 
\end{cor}

A.\ Rauf \cite{asia} proved that $\depth_S(S/(I:u))\geq \depth_S(S/I)$, for any monomial ideal $I\subset S$ and any monomial $u\in S$, see \cite[Corollary 1.3]{asia}. Similar results hold for $\sdepth_S(I:u)$ and $\sdepth_S(S/(I:u))$. In order to show that, we use Corollary $2.5$ and the following result from \cite{mir}.

\begin{teor}\cite[Theorem 1.4]{mir}
Let $I \subset S$ be a monomial ideal such that $I = v(I : v)$,
for a monomial $v \in S$. Then $\sdepth_S(I) = \sdepth_S(I : v)$, $\sdepth_S(S/I) = \sdepth_S(S/(I : v))$.
\end{teor}

\begin{prop}
If $I\subset S$ is a monomial ideal and $u\in S$ a monomial, then:

(1) $\sdepth_S(I:u)\geq \sdepth_S(I)$. (\cite[Proposition 1.3]{pop})

(2) $\sdepth_S(S/(I:u))\geq \sdepth_S(S/I)$. 
\end{prop}

\begin{proof}
(1) Note that $I\cap(u) = u(I:u)$. By Theorem $2.6$, it follows that $\sdepth_S(I:u) = \sdepth_S(I\cap(u))\geq
    \sdepth_S(I)$. See another proof in \cite{pop}.
    
(2) By Theorem $2.6$ and Corollary $2.5$, $\sdepth_{S}(S/(I:u))= \sdepth_S(S/(I\cap(u))$. 
\end{proof}

Note that if $P\in Ass(S/I)$ is an associated prime, then there exists a monomial $v\in S$ such that $P=(I:v)$. Using the above Proposition, we obtain again the results of Ishaq \cite{ishaq} and Apel \cite{apel} .

\begin{cor}
If $I\subset S$ is a monomial ideal, with $Ass(S/I)=\{P_1,\ldots,P_r\}$. If we denote $d_i=|P_i|$, we have:

(1) $\sdepth_S(I)\leq \min\{ n - \left\lfloor d_i/2 \right\rfloor:\; i=1,\ldots r \}$. (Ishaq)

(2) $\sdepth_S(S/I)\leq \min\{ n - d_i:\; i=1,\ldots r\}$. (Apel)
\end{cor}

\begin{proof}
(1) It is enough to notice that $\sdepth_S(P_i)= n - \left\lfloor d_i/2 \right\rfloor$. See also \cite[Theorem 1.1]{ishaq}.

(2) It is enough to notice that $\sdepth_S(P_i)= n - d_i$. See also \cite{apel}.
\end{proof}

\begin{cor}
Let $I\subset S$ be a monomial ideal minimally generated by $m$ monomials, such that there exists a prime ideal $P \in Ass(S/I)$ with $ht(P) = m$. Then $\sdepth_S(S/I) = n-m$.
\end{cor}

\begin{proof}
It is a direct consequence of Theorem $2.6$ and Corollary $2.8(2)$.
\end{proof}

\begin{obs}
\emph{Let $I\subset S$ be a monomial ideal. Then $\sdepth_S(S/I)=n-1$ if and only if $I$ is principal.
Indeed, $I$ is principal if and only if all the primes in $Ass(S/I)$ have height $1$. Therefore, we are done by Corollary $2.8(2)$.}
\end{obs}

\begin{cor}
Let $k\geq 2$ be an integer, and let $I_j\subset S$ be some monomial ideals, where $1\leq j\leq k$. Then:

(1) $\sdepth_S(I_1 \cap \cdots \cap I_k) \geq \sdepth_{S}(I_1) + \cdots + \sdepth_{S}(I_k) - n(k-1)$.

(2) $\sdepth_S(I_1 + \cdots +I_k) \geq \min \{\sdepth_{S}(I_1), \sdepth_{S}(I_2)+\sdepth_{S}(S/I_1) - n, \ldots,\linebreak 
\sdepth_{S}(I_k) + \sdepth_{S}(S/I_{k-1}) + \cdots + \sdepth_{S}(S/I_{1}) - n(k-1) \}$.

(3) $\sdepth_S(S/(I_1\cap \cdots \cap I_k)) \geq \min \{\sdepth_{S}(S/I_1), \sdepth_{S}(S/I_2)+\sdepth_{S}(I_1) - n, \ldots, \sdepth_{S}(S/I_k) + \sdepth_{S}(I_{k-1}) + \cdots + \sdepth_{S}(I_{1})-n(k-1)\}$.

(4) $\sdepth_S(S/(I_1+\cdots+I_k)) \geq \sdepth_{S}(S/I_1) + \cdots + \sdepth_{S}(S/I_k) - n(k-1)$.
\end{cor}

\begin{proof}
We use induction on $k\geq 2$ and we apply Theorem $2.2$.
\end{proof}

\begin{cor}
Let $I,J \subset S$ be two monomial ideals, such that $G(J)=\{u_1,\ldots,u_k\}$ is the set of minimal monomial generators of $J$. Then:

(1) $\sdepth_S(I:J) \geq \sdepth_S(I:u_1) + \sdepth_S(I:u_2) + \cdots + \sdepth_S(I:u_k) - n(k-1) \geq k\sdepth_S(I) - n(k-1)$.

(2) $\sdepth_S(S/(I:J)) \geq \min \{\sdepth_{S}(S/(I:u_1)), \sdepth_{S}(S/(I:u_2))+\sdepth_{S}(I:u_1) - n, \ldots, \sdepth_{S}(S/(I:u_k)) + \sdepth_{S}(I:u_{k-1}) + \cdots + \sdepth_{S}(I:u_1)-n(k-1)\} \geq
\sdepth_{S}(S/I) + (k-1)\sdepth_S(I) - n(k-1)$.
\end{cor}

\begin{proof}
(1) Note that $(I:J) = (I:u_1)\cap (I:u_2)\cap \cdots \cap (I:u_k)$. Therefore, the first inequality is a direct consequence of $2.11(1)$. The second inequality is a consequence of Proposition $2.7(1)$.

(2) Similarly to (1), we use Corollary $2.11(3)$ and Proposition $2.7(2)$.
\end{proof}

Now, let $I\subset S$ be a monomial ideal and let $I=C_1\cap \cdots \cap C_k$, be the irredundant minimal decomposition of $I$. If we denote $P_j=\sqrt{C_j}$ for $1\leq j\leq k$, we have $Ass(S/I)=\{P_1,\ldots,P_k\}$. 
In particular, if $I$ is squarefree, $C_j=P_j$ for all $j$. Denote $d_j=|P_j|$, where $1\leq i\leq k$. We may assume that $d_1\geq d_2 \geq \cdots \geq d_k$.
Using \cite[Theorem 1.3]{mirc}, Proposition $2.8$ and Corollary $2.11$, we obtain, by straightforward computations, the following bounds for $\sdepth_S(I)$ and $\sdepth_S(S/I)$.

\begin{cor}
(1) $n - \left\lfloor d_1/2 \right\rfloor \geq \sdepth_S(I) \geq n - \left\lfloor d_1/2 \right\rfloor - \cdots - \left\lfloor d_k/2 \right\rfloor$.

(2) $n - d_1 \geq \sdepth_S(S/I) \geq n - \left\lfloor d_1/2 \right\rfloor - \cdots - \left\lfloor d_{k-1}/2 \right\rfloor - d_k$.
\end{cor}

In a more general case, let $I=Q_1\cap \cdots \cap Q_k$ be the primary irredundant decomposition of $I$,  $P_i=\sqrt{Q_i}$ and denote $q_j=\sdepth_S(Q_j)$ and $d_j=|P_j|$. 
We may assume that $d_1\geq d_2 \geq \cdots \geq d_k$.
Note that $q_j\leq n - d_j/2$, since $P_j = (Q_j:u_j)$, where $u_j\in S$ is a monomial, and therefore $\sdepth_S(Q_j)\leq \sdepth_S(P_j)$, by Proposition $2.7(1)$. On the other hand, we obviously have $\sdepth_{S}(S/Q_j)=\sdepth_S(S/P_j)$.
Using Proposition $2.8$ and Corollary $2.11$, we obtain, by straightforward computations, the following bounds for $\sdepth_S(I)$ and $\sdepth_S(S/I)$.

\begin{cor}
(1) $n - \left\lfloor d_1/2 \right\rfloor \geq \sdepth_S(I) \geq q_1+\cdots+q_k-n(k-1)$.

(2) $n - d_1 \geq \sdepth_S(S/I) \geq \min\{n-d_1, q_1-d_2, q_1+q_2-d_3-n, \ldots,\linebreak q_1+\cdots+q_{k-1}-d_k-n(k-2)\}$.
\end{cor}

\begin{exm}
Let $I=Q_1\cap Q_2\cap Q_3 \subset S:=K[x_1,\ldots,x_7]$, where $Q_1=(x_1^2,\ldots,x_5^2)$, $Q_2=(x_4^3,x_5^3,x_6^3)$
and $Q_3=(x_6^3,x_6x_7,x_7^2)$. Denote $P_j=\sqrt{Q_j}$. Note that $q_3=\sdepth_S(Q_3)=\sdepth_{K[x_6,x_7]}(Q_3\cap K[x_6,x_7])+5 = 1+5 = 6$. Also, since $Q_1$ and $Q_2$ are generated by powers of variables, by \cite[Theorem 1.3]{mirc},  $q_1=7-\left\lfloor 5/2 \right\rfloor = 5$ and $q_2 = 7 - \left\lfloor 3/2 \right\rfloor = 6$. According to Corollary $2.14$, we have $5=7-\left\lfloor d_1/2 \right\rfloor \geq \sdepth_S(I)\geq q_1+q_2+q_3-14 = 3$ and $2=7-d_1 \geq \sdepth_S(S/I)\geq  \min\{7-d_1, q_1-d_2, q_1+q_2-d_3-7\} = \min\{7-5,5-3,5+6-2-7\}=2$. Thus $\sdepth_S(I)\in \{3,4,5\}$ and $\sdepth_S(S/I)=2$.

On the other hand, $\depth_S(S/I)\leq \min\{n-\depth_S(S/P_j):\;j=1,2,3\} = 2$. In particular, we have $\sdepth_S(I)\geq \depth_S(I)$ and $\sdepth_S(S/I)\geq \depth_S(S/I)$. Thus both $I$ and $S/I$ satisfy the Stanley conjecture. In fact, using CoCoA, we get $\depth_S(S/I)=2$. 

\end{exm}

\section{Equivalent forms of Stanley conjecture}

\begin{prop}
The following assertions are equivalent:

(1) For any integer $n\geq 1$ and any monomial ideal $I\subset S=K[x_1,\ldots,x_n]$, Stanley conjecture holds for $I$, i.e. $\sdepth_S(I)\geq \depth_S(I)$.

(2) For any integer $n\geq 1$ and any monomial ideals $I,J\subset S$, if $\sdepth_S(I+J)\geq \depth_S(I+J)$, then
$\sdepth_S(I)\geq \depth_S(I)$.

(3) For any integers $n,m\geq 1$, any monomial ideal $I\subset S=K[x_1,\ldots,x_n]$, if $u_1,\ldots,u_m\in S$ is a regular sequence on $S/I$ and $J=(u_1,\ldots,u_m)$, then if:
\[ \sdepth_S(I+J) \geq \depth_S(I+J) \Rightarrow \sdepth_S(I) \geq \depth_S(I). \]

(4) For any integers $n,m\geq 1$, any monomial ideal $I\subset S=K[x_1,\ldots,x_n]$, if $u_1,\ldots,u_m\in S$ is a regular sequence on $S/I$ and $J=(u_1,\ldots,u_m)$, then if:
\[ \sdepth_S(I+J) = \depth_S(I+J) \Rightarrow \sdepth_S(I) = \depth_S(I). \]

(5) For any integer $n\geq 1$, any monomial ideal $I\subset S=K[x_1,\ldots,x_n]$, if $\bar{S}=S[y]$, then:
$\sdepth_{\bar S}(I,y) = \depth_S(I) \Rightarrow \sdepth_S(I) = \depth_S(I)$.
\end{prop}

\begin{proof}
$(1)\Rightarrow(2)\Rightarrow(3)$. Are obvious.

$(3)\Rightarrow(4)$. Assume $\sdepth_S(I+J)=\depth_S(I+J)$. Note that $\depth_S(I+J)=\depth_S(I)-m$, since $u_1,\ldots,u_m\in S$ is a regular sequence on $S/I$. By Corollary $2.4(2)$, $\sdepth_S(I+J)\geq \sdepth_S(I)-m$. Since $\sdepth_S(I)\geq \depth_S(I)$ by (3), we get $\sdepth_S(I)=\depth_S(I)$.

$(4)\Rightarrow(5)$. It is obvious, since $y$ is regular on $\bar{S}/I\bar{S}$ and we apply $(4)$ for $I\bar{S}$.

$(5)\Rightarrow(1)$. Let $I\subset S$ be a monomial ideal. Assume by contradiction that $\sdepth_S(I)<\depth_S(I)$. If $k\geq 1$ is an integer, we denote $I_k=(I,y_1,\ldots,y_k)\subset S_k:=S[y_1,\ldots,y_k]$. Note that $y_1,\ldots,y_k$ is a regular sequence on $S_k/I_k$ and therefore $\depth_{S_k}(I_k) = \depth_S(I)$. According to Corollary $1.4(1)$, we have: $$\sdepth_{S_k}(I_k) \geq \min\{\sdepth_S(I)+k,\; \sdepth_S(S/I)+ \left\lceil k/2 \right\rceil \}.$$ 
It follows that there exists $k_0\leq 1$, such that $\sdepth_{S_k}(I_k) \geq \depth_S(I)$ for any $k\geq k_0$. If we chose $k_0$ minimal with this property, we claim that $\sdepth_{S_{k_0}}(I_{k_0}) = \depth_S(I)$. Indeed, it is enough to notice that $\sdepth_{S_k}(I_k)\leq \sdepth_{S_{k-1}}(I_{k-1})+1$. Now, by applying $(5)$ inductively, it follows that $\sdepth_S(I) = \depth_S(I)$, a contradiction.
\end{proof}

\begin{obs}
\emph{Let $I\subset S=K[x_1,\ldots,x_n]$ be a monomial ideal such that $\sdepth_S(I)\geq \depth_S(I)$. Let $u_1,\ldots,u_m\in S$ be a regular sequence on $S/I$ and $J=(u_1,\ldots,u_m)$. Note that $\depth_S(I\cap J) = \depth_S(I+J)+1 = \depth_S(I)-m+1$. Also, by Corollary $2.4(1)$, we have $\sdepth_S(I\cap J) \geq \sdepth_S(I) - \left\lfloor m/2 \right\rfloor$. Assume $\sdepth_S(I\cap J) = \depth_S(I\cap J)$. It follows that 
$ \depth_S(I)-m+1 \geq \sdepth_S(I) - \left\lfloor m/2 \right\rfloor \geq \depth_S(I) - \left\lfloor m/2 \right\rfloor \geq \depth_S(I)-m+1$.}

\emph{Therefore, $\sdepth_S(I)=\depth_S(I)$ and $\left\lfloor m/2 \right\rfloor = m-1$, and thus $m\leq 2$. In particular, if we could find an ideal $I\subset S$ such that, by denoting $\bar S=S[y_1,y_2,y_3]$, if \linebreak
$\sdepth_{\bar S}(I\bar S\cap (y_1,y_2,y_3))=\depth_S(I)$, we contradict the Stanley conjecture for $I$.}
\end{obs}

\begin{prop}
The following assertions are equivalent:

(1) For any integer $n\geq 1$ and any monomial ideal $I\subset S=K[x_1,\ldots,x_n]$, Stanley conjecture holds for $I$, i.e. $\sdepth_S(I)\geq \depth_S(I)$.

(2) For any integer $n\geq 1$ and any monomial ideals $I,J\subset S$, if $\sdepth_S(I\cap J)\geq \depth_S(I\cap J)$ then $\sdepth_S(I)\geq \depth_S(I)$.

(3) For any integers $n,m\geq 1$, any monomial ideal $I\subset S=K[x_1,\ldots,x_n]$, if $u_1,\ldots,u_m\in S$ is a regular sequence on $S/I$ and $J=(u_1,\ldots,u_m)$, then:
$$\sdepth_S(I\cap J)\geq \depth_S(I\cap J) \Rightarrow \sdepth_S(I)\geq \depth_S(I).$$
\end{prop}

\begin{proof}
$(1)\Rightarrow(2)$ and $(2)\Rightarrow(3)$. There is nothing to prove. 

$(3)\Rightarrow(1)$. Let $I\subset S$ be a monomial ideal. Assume by contradiction that $\sdepth_S(I)<\depth_S(I)$. 
For any integer $k\geq 1$, we define $I_k:=(I,y_1,\ldots,y_k)\subset S_k:=S[y_1,\ldots,y_k]$. Denote $J=(y_1,\ldots,y_k)\subset S_k$. Note that $y_1,\ldots,y_k$ is a regular sequence on $S_k/IS_k$. By Corollary $2.4(1)$, we have $\sdepth_{S_k}(I_k)\geq \sdepth_S(I)+\left\lceil k/2 \right\rceil$. On the other hand, by Corollary $1.4(5)$, $\depth_{S_k}(I_k) = \depth_S(I)$. It follows that there
exists a $k_0\geq 1$, such that $\sdepth_{S_k}(I_k)\geq \depth_{S_k}(I_k)$ for any $k\geq k_0$, and therefore, by (2), we get $\sdepth_S(I)\geq \depth_S(I)$, as required.
\end{proof}

\begin{prop}
The following assertions are equivalent:

(1) For any integer $n\geq 1$ and any monomial ideal $I\subset S=K[x_1,\ldots,x_n]$, Stanley conjecture holds for $S/I$, i.e. $\sdepth_S(S/I)\geq \depth_S(S/I)$.

(2) For any integer $n\geq 1$ and any monomial ideals $I,J\subset S$, if $\sdepth_S(S/(I\cap J))\geq \depth_S(S/(I\cap J))$ then $\sdepth_S(S/I)\geq \depth_S(S/I)$.

(3) For any integers $n,m\geq 1$, any monomial ideal $I\subset S=K[x_1,\ldots,x_n]$, if $u_1,\ldots,u_m\in S$ is a regular sequence on $S/I$ and $J=(u_1,\ldots,u_m)$, then:
$$\sdepth_S(S/(I\cap J))\geq \depth_S(S/(I\cap J)) \Rightarrow \sdepth_S(S/I)\geq \depth_S(S/I).$$
\end{prop}

\begin{proof}
$(1)\Rightarrow(2)$ and $(2)\Rightarrow(3)$. There is nothing to prove.

$(3)\Rightarrow(1)$. Let $I\subset S$ be a monomial ideal. Assume by contradiction that $\sdepth_S(I)<\depth_S(I)$. 
For any integer $k\geq 1$, we define $I_k:=(I,y_1,\ldots,y_k)\subset S_k:=S[y_1,\ldots,y_k]$. Note that $y_1,\ldots,y_k$ is a regular sequence on $S_k/IS_k$. By Corollary $2.4(4)$, $\sdepth_{S_k}(S_k/I_k)\geq 
\min\{n, \sdepth_S(S/I)+\left\lceil k/2 \right\rceil \}$. On the other hand, by Corollary $1.4(5)$, $\depth_{S_k}(S_j/I_k) = \depth_S(S/I)$. It follows that there
exists a $k_0\geq 1$, such that $\sdepth_{S_k}(I_k)\geq \depth_{S_k}(I_k)$ for any $k\geq k_0$, and therefore, by (2), we get $\sdepth_S(I)\geq \depth_S(I)$, as required.
\end{proof}

\begin{obs}
\emph{Let $I\subset S=K[x_1,\ldots,x_n]$ be a monomial ideal such that $\sdepth_S(S/I)\geq \depth_S(S/I)$. Let $u_1,\ldots,u_m\in S$ be a regular sequence on $S/I$ and $J=(u_1,\ldots,u_m)$. Note that $\depth_S(S/(I\cap J)) = \depth_S(S/(I+J))+1 = \depth_S(S/I)-m+1$. Also, by Corollary $2.4(4)$, we have $\sdepth_S(S/(I\cap J)) \geq \min\{ n-m, \sdepth_S(S/I) - \left\lfloor m/2 \right\rfloor \}$ 
Assume $\sdepth_S(S/(I\cap J)) = \depth_S(S/(I\cap J))$.} 

\emph{It follows that $ \depth_S(S/I)-m+1 \geq \min\{ n-m, \sdepth_S(S/I) - \left\lfloor m/2 \right\rfloor \} \geq \linebreak \min\{ n-m, \depth_S(S/I) - \left\lfloor m/2 \right\rfloor \} \geq \min\{n-m, \depth_S(S/I)-m+1\} = \depth_S(S/I)-m+1$ and therefore, we have equalities.}

\emph{If $I$ is principal, then $\depth_S(S/I)=n-1$ and therefore $\min\{ n-m, \depth_S(S/I) - \left\lfloor m/2 \right\rfloor \} = n-m$. It follows that $\depth_S(S/I) - \left\lfloor m/2 \right\rfloor = n-1-\left\lfloor m/2 \right\rfloor  \geq n-m$ which is true for all $m$. If $I$ is not principal, then by Remark $2.10$, $\depth_S(S/I)\leq n-2$. It follows that $\min\{ n-m, \sdepth_S(S/I) - \left\lfloor m/2 \right\rfloor \} = \sdepth_S(S/I) - \left\lfloor m/2 \right\rfloor = \depth_S(S/I)-m+1$. Therefore, $\sdepth_S(S/I)=\depth_S(S/I)$ and $m\leq 2$.}

\emph{In particular, if we could find an ideal $I\subset S$ which is not principal, such that, denoting $\bar S=S[y_1,y_2,y_3]$, if $\sdepth_{\bar S}(\bar S/(I \bar S\cap (y_1,y_2,y_3)))=\depth_S(I)$, we contradict the Stanley conjecture for $S/I$.}
\end{obs}

\begin{lema}
Let $I\subset J \subset S=K[x_1,\ldots,x_n]$ be two monomial ideals and denote $\bar S:=S[y]$. Then:
\[ \sdepth_S(J/I)+1 \geq \sdepth_{\bar S}((J\bar S+(y))/I\bar S) \geq \min\{\sdepth_S(J/I), \sdepth_S(S/I)+1 \}. \]
\end{lema}

\begin{proof}
In order to prove the first inequality, we consider $\bigoplus_{i=1}^r u_iK[Z_i]$, a Stanley decomposition of $(J\bar S+(y))/I\bar S$. Note that $((J\bar S+(y))/I\bar S)\cap S = J/I$ and therefore, $J/I=\bigoplus_{y\nmid u_i} u_iK[Z_i\setminus \{y\}]$ is a Stanley decomposition. 

The second inequality follows from the fact that $(J\bar S+(y))/I\bar S = J/I\oplus y(S/I)[y]$.
\end{proof}

As a particular case of Example $1.10$, we consider the following Lemma.

\begin{lema}
Let $J=(x_1,\ldots,x_n)\cap(y_1,\ldots,y_m)\subset S'=K[x_1,\ldots,x_n,y_1,\ldots,y_m]$ with $n\geq m$. Then:

(1) $m \geq \sdepth_{S'}(S'/J) \geq \min\{m, \left\lceil  n/2 \right\rceil \}$.

(2) $\depth_{S'}(S'/J) = 1$.

In particular, if $n\geq 2m-1$, then $\sdepth_{S'}(S'/J)=m$.
\end{lema}

\begin{prop}
The following assertions are equivalent:

(1) For any integer $n\geq 1$ and any monomial ideal $I\subset S=K[x_1,\ldots,x_n]$, Stanley conjecture holds for $S/I$ and $I$.

(2) For any integer $n\geq 1$ and any monomial ideals $I,J\subset S$ with $\supp(I)\cap\supp(J)=\emptyset$, we have:
If $\sdepth_S((I+J)/I) = \depth_S((I+J)/I)$, then $\sdepth_S(S/I) = \depth_S(S/I)$ and $\sdepth_S(J) \geq \depth_S(J)$.
\end{prop}

\begin{proof}
$(1)\Rightarrow(2)$. Let $I,J\subset S$ be two monomial ideals, with $\supp(I)\cap\supp(J)=\emptyset$, and assume 
$\sdepth_S((I+J)/I) = \depth_S((I+J)/I)$. According to Theorem $1.2(6)$, we have $\depth_S((I+J)/I)=\depth_S(I+J)=\depth_S(S/I)+\depth_S(J)-n$. On the other hand, by Remark $1.3$, $\sdepth_S((I+J)/I)\geq \sdepth_S(S/I)+\sdepth_S(J)-n$. By (1), it follows that $\sdepth_S(S/I) = \depth_S(S/I)$ and $\sdepth_S(J) = \depth_S(J)$. In particular, $\sdepth_S(J) \geq \depth_S(J)$.

$(2)\Rightarrow(1)$. Let $I\subset S$ be a monomial ideal. For any positive integer $k$, we denote $S_k=S[y_1,\ldots,y_k]$ and $I_k=(I,y_1,\ldots,y_k)\subset S_k$. Assume $\sdepth_S(S/I)<\depth_S(S/I)$. Since $\sdepth_{S_k}(I_k/IS_k) \geq \sdepth_S(S/I)+ \left\lfloor k/2 \right\rfloor$, it follows that there exists a positive integer $k_0$ such that $\sdepth_{S_{k}}(I_{k}/IS_{k}) \geq \depth_{S_{k}}(I_{k}/IS_{k}) = \depth_S(S/I),\;\;(\forall)k\geq k_0 \;(*).$
If we apply Lemma $3.6$ for $I_k\subset S_k$ and $y_{k+1}$, we obtain $\sdepth_{S_{k+1}}(I_{k+1}/IS_{k+1}) \leq  \sdepth_{S_k}(I_k/IS_k)+1$. Thus, if we chose the minimal $k_0$ with the property $(*)$, we have in fact
$\sdepth_{S_{k_0}}(I_{k_0}/IS_{k_0}) = \depth_S(S/I)$. By $(2)$, it follows that $\sdepth_S(S/I)=\depth_S(S/I)$, a contradiction.

Now, assume $\sdepth_S(I)<\depth_S(I)$, and denote $J_k=(y_1,\ldots,y_{2k-1})\cap (y_{2k},\ldots,y_{3k-1})\subset S_k:=S[y_1,\ldots,y_{3k-1}]$. According to Lemma $3.7$, we have $\sdepth_{S_k}(S_k/J_k)=n+k$ and $\depth_{S_k}(S_k/J_k) =1$. Let $I_k:=IS_k+J_k$. By Remark $1.3$, $\sdepth_{S_k}(I_k/J_k) \geq \sdepth_{S}(I) + k$. On the other hand $\depth_{S_k}(I_k/J_k) = \depth_{S}(I) + \depth_{S_k}(S_k/J_k)-n = \depth_S(I) + 1$.

Therefore, there exists a positive integer $k_0$, such that $\sdepth_{S_k}(I_k/J_k) \geq \sdepth_{S_k}(I_k/J_k)$ for any $k\geq k_0$. It follows, by $(2)$, that $\sdepth_S(I)\geq \depth_S(I)$, a contradiction.
\end{proof}

\end{document}